\documentclass[12pt]{article}
\usepackage{amsmath} 
\usepackage{graphicx} 
\usepackage{latexsym} 
\usepackage{amsfonts} 
\usepackage{amssymb} 
\usepackage{fancyhdr} 
\usepackage{hyperref} 
\usepackage{xcolor} 
\usepackage{pgfplots} 
\usepackage{amsthm} 
\usepackage{bbm}
\usepackage[nottoc]{tocbibind}
\usepackage{autonum}
\usepackage{hyperref}
\usepackage{srcltx}
\usepackage[english]{babel}

\newtheorem{thm}{Theorem}[section]
\newtheorem{prop}[thm]{Proposition}
\newtheorem{remark}[thm]{Remark}

\newtheorem{cor}[thm]{Corollary}

\newtheorem{lemma}[thm]{Lemma}

\newtheorem*{thm:rescaledhullslines}{Proposition \ref{thm:rescaledhullslines}}
\newtheorem*{thm:timechangesqrtnotsimple}{Theorem \ref{thm:timechangesqrtnotsimple}}
\newtheorem*{thm:timechangebmnotsimple}{Theorem \ref{thm:timechangebmnotsimple}}
\newtheorem*{thm:tangentialdeparture}{Proposition \ref{tangentialdeparture}}

\newcommand{\be}{\begin{equation}}
\newcommand{\ee}{\end{equation}}
\newcommand{\bes}{\begin{equation*}}
\newcommand{\ees}{\end{equation*}}

\renewcommand{\l}{\lambda}

\newcommand{\g}{\gamma}

\newcommand{\R}{\mathbb R}

\newcommand{\e}{\epsilon}

\newcommand{\hcap}{\operatorname{hcap}}

\def\E{{\mathbb E}}
\def\P{{\mathbb P}}

\def\1{{\bf 1}}

\begin{document}

\title{{\Large \bf Effect of random time changes on Loewner hulls\thanks{\scriptsize To appear in Revista Matem\'atica Iberoamericana}}}
\author{{\bf Kei Kobayashi\thanks{\scriptsize Department of Mathematics, Fordham University, New York, NY 10023, USA. Email: kkobayashi5@fordham.edu},
		Joan Lind\thanks{\scriptsize Department of Mathematics, The University of Tennessee, Knoxville, TN 37996, USA. Email: jlind@utk.edu}, and 
		Andrew Starnes\thanks{\scriptsize Department of Mathematics, University of Hartford, West Hartford, CT 06117, USA. Email: starnes@hartford.edu}}}

\date{October 12, 2019}

\maketitle

%
%

\begin{abstract}
Loewner hulls are determined by their real-valued driving functions.  
We study the geometric effect on the Loewner hulls when the driving function is composed with a random time change, such as the inverse of an $\alpha$-stable subordinator.  
In contrast to SLE, we show that for a large class of random time changes, the time-changed Brownian motion process does not generate a simple curve.  
Further we develop criteria which can be applied in many situations to determine whether the Loewner hull generated by a time-changed driving function is simple or non-simple.
To aid our analysis of an example with a  time-changed deterministic driving function,
we prove a deterministic result that a driving function that moves faster than $at^r$ for $r \in (0,1/2)$ generates a hull that leaves the real line tangentially.\par\vspace{3mm}

\noindent\textit{key words:} Loewner evolution, random time change, inverse subordinator, time-changed Brownian motion
\end{abstract}

%
%


\section{Introduction}

Schramm--Loewner Evolution, denoted SLE$_\kappa$, is a family of random curves in the upper halfplane $\mathbb{H}$ that is generated by the random function $\lambda(t)=\sqrt{\kappa}B_t$, where $\kappa\geq0$ and $(B_t)_{t\geq 0}$ is a one-dimensional standard Brownian motion. Based on $\kappa$, SLE$_\kappa$ exhibits phase transitions. Namely, for $0\leq\kappa\leq4$ the curves are simple, for $4<\kappa<8$ the curves are not simple, and for $8\leq\kappa$ the curves are spacefilling; 
where all of these hold almost surely \cite{basicpropertiesofsle}.
In the deterministic setting, it is also known that if we let $\lambda(t)$ be a H\"older-$\frac{1}{2}$ continuous function with norm $\|\lambda\|_{1/2}$, then there are similar phase transitions: for $0\leq\|\lambda\|_{1/2}<4$ the curves generated are simple \cite{sharpcondition}, and for $\|\lambda\|_{1/2}<4.0001$ the curves generated are not spacefilling \cite{spacefilling}. It is our goal to analyze what happens to the curves when some of these functions are composed with a continuous, non-decreasing stochastic process $(E_t)_{t\ge 0}$, called a random time change.

Among the simplest yet most important random time changes is the so-called inverse $\alpha$-stable subordinator, where $\alpha\in(0,1)$ is a parameter. With this specific time change $(E_t)$ assumed independent of $(B_t)$, the time-changed Brownian motion $B\circ E=(B_{E_t})_{t\ge 0}$ has been widely used to model subdiffusions, where particles spread at a slower rate than the usual Brownian particles. Indeed, the variance $\mathbb{E}[(B_{E_t})^2]$ takes the form $c_\alpha t^\alpha$, which grows more slowly for large $t$ than the variance of the Brownian motion. 
More detailed backgrounds and relevant references about SLE$_\kappa$ and random time changes are provided in Section \ref{background}.

Our main result reveals the fact that the time change extremely modifies the original curves, and phase transitions do not occur.  See Figure \ref{fig:BEt} which compares a sample curve in this particular case with a sample curve in the untime-changed setting.

\begin{thm}\label{thm:timechangebmnotsimple}
For any $\kappa >0$, almost surely the time-changed Brownian motion process $(\kappa B_{E_t})_{t\ge 0}$ does not generate a simple curve.
\end{thm}
To prove this, we use a result in \cite{lindrobins} to first derive general criteria (Theorem \ref{thm:timechangesqrtnotsimple}) for verifying whether the curves generated by time-changed functions are simple or non-simple. 
The proof of Theorem \ref{thm:timechangebmnotsimple}
also relies on a deep relationship between Brownian motion and a 3-dimensional Bessel process given in \cite{williams} and local behaviors of Bessel processes studied in \cite{shiga}. 
\begin{figure}
    \centering
    \includegraphics[width=4.9in]{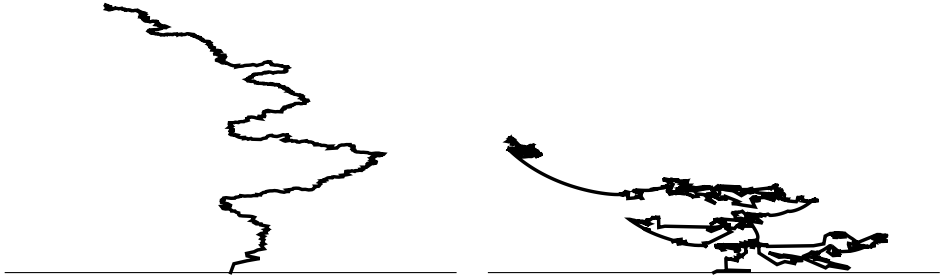}
    \caption{Sample SLE$_1$ curve (left) and sample curve generated by a time-changed Brownian motion (right).}
    \label{fig:BEt}
\end{figure}

We also investigate the scaling limits of random curves generated by time-changed self-similar processes.  In particular, Corollary \ref{cor:rescaledhullslines} shows that rescaling the curves generated by $\lambda(t)=\kappa B_{E_t}$ leads to deterministic sets, as observed in \cite{chenrohde} for curves generated by a symmetric stable process (without a time change). 

To further understand the effect of the random time change, we explore some examples of
curves generated by time-changed deterministic functions, including a time-changed Weierstrass function.  To aid our analysis of the deterministic examples, we also derive a condition on $\lambda(t)$ that guarantees that the generated curves leave the real line tangentially:

\begin{prop}\label{tangentialdeparture}
Suppose that $\l(0)=0$ and $\l(t) \geq a t^r$  where $a>0$ and $r \in (0, 1/2]$.  
Then for $t$ small enough, the  hull $K_t$ driven by $\l$ is contained in the region 
$\{ x+iy \, : \, 0\leq x, \, 0<y< \frac{26}{a} \,x^{2-2r} \}$.
\end{prop}

This proposition may be of independent interest, as it provides a converse of sorts to recent work by Lau and Wu \cite{lauwu}.
In particular, Lau and Wu show that if the curve leaves the real line tangentially by lying in the domain $\{ x+iy \, : \, 0< x, \, ax^r < y < bx^r \}$ for $r>1$, then we have some control on the associated driving function, namely $\limsup_{t \to 0} t^{-1/(r+1)} |\l(t)| < \infty.$  They also analyze a particular family of tangential curves, which generalizes a result  of Prokhorov and Vasil'ev \cite{prokhorovvasilev}.

Generalizations of SLE$_\kappa$ to the case of the time-changed Brownian motion are considered and numerically analyzed in \cite{Nezhadhaghighi2011,credidio2015}. However, as far as we know, our investigation in this paper provides the first theoretical account of geometric properties of random curves associated with a large class of time-changed functions.  

Time change is not the only adaptation of SLE that has been considered.
Another example of an SLE variant is found in \cite{rushkinetal} and  \cite{guanwinkel}, where an $\alpha$-stable L\'evy process is added to the Brownian motion, which adds jumps to the driving function. It was shown that the phases of the hulls were unchanged by this addition, but the spread of the hull along the real line changes based on $\alpha$ due to the jumps.  

We end this section with comments on the organization of the paper. 
 In Section \ref{background}, we discuss the Loewner equation and random time changes.  
 Section \ref{scalingarguments} contains results on rescaled hulls and the proof of Proposition \ref{tangentialdeparture}.
 In the last section, we state and prove the criteria (Theorem \ref{thm:timechangesqrtnotsimple}) for verifying whether the curves generated by time-changed functions are simple or non-simple, we prove Theorem \ref{thm:timechangebmnotsimple}, and we discuss examples of time-changed deterministic driving fucntions.


\section{Background}\label{background}

\subsection{Loewner Equation}

Let $\mathbb{H}=\{x+iy\in \mathbb{C}: y> 0\}$ denote the upper halfplane. 
Let $\lambda:[0,T]\to\mathbb{R}$ be continuous. The \textit{(chordal) Loewner equation} is given by
\begin{equation}\label{eqn:le}
    \frac{\partial}{\partial t}g_t(z)=\frac{2}{g_t(z)-\lambda(t)},\quad g_0(z)=z
\end{equation}
for $z\in\overline{\mathbb{H}}\setminus\{\lambda(0)\}$. We call $\lambda$ the \textit{driving function} of $g_t$. For $z\in\overline{\mathbb{H}}\setminus\{\lambda(0)\}$, there is a time interval so that (\ref{eqn:le}) has a solution, and we define $T_z$ to be the maximum such time, i.e., $T_z=\sup\{s\in[0,T):g_t(z)\text{ exists on }[0,s)\}$.  
So that it is defined on all of $\overline{\mathbb{H}}$, we set $T_{\lambda(0)} = 0$.
Define $K_t=\{z\in\overline{\mathbb{H}}:T_z\leq t\}$, meaning that $K_t$ is the collection of points so that $g_s(z)=\lambda(s)$ for some $s\leq t$.  We call $K_t$ the \textit{(Loewner) hull} generated (or driven) by $\lambda$. It can be shown that $\mathbb{H}\setminus K_t$ is simply connected and $g_t:\mathbb{H}\setminus K_t\to\mathbb{H}$ is conformal. Furthermore, $g_t$ is the unique conformal map with following expansion (called the \textit{hydrodynamic normalization}) near infinity:
\begin{equation}\label{eqn:hydronorm}
    g_t(z)=z+\frac{c(t)}{z}+O\left(\frac{1}{z^2}\right).
\end{equation}
One can further show that $c(t)=2t$.  This quantity is useful as it tells us about the size of the hull as viewed from infinity, and so we define the \textit{halfplane capacity} of $K_t$ as $\hcap(K_t) = c(t)/2=t$. 
It has the following probabilistic interpretation:
$$\hcap(K_t) = \frac{1}{2} \lim_{y \to \infty} y \, \E^{iy}{ \text{Im}(B_\tau)},$$
where $B_t$ is a Brownian motion started at $iy$ and stopped at $\tau = \inf\{ s \, : \, B_s \in \R \cup K_t \}$  
(see Proposition 3.41 in \cite{lawler}).
When the limit exists, define
\begin{equation}
    \gamma(t)=\lim_{y\downarrow 0}g_t^{-1}(\lambda(t)+iy).
\end{equation}
If $\gamma(t)$ exists and is continuous for every $t\in[0,T]$, we call $\gamma(t)$ the trace of $K_T$. 
When $\gamma(0,T]$ is a simple curve in $\mathbb{H}$, then $K_t=\gamma[0,t]$, and we say that $\lambda$ generates a simple curve.  It is possible that the trace $\gamma$ exists, but $\lambda$ does not generate a simple curve.  In this case, 
$K_t$ is the closure in $\overline{\mathbb{H}}$ of the complement of the unbounded component of $\mathbb{H}\setminus\gamma[0,t]$. Intuitively, this means that the boundary of $K_t$ is governed by $\gamma(t)$.

On the flipside, if we were to start with a family of continuously growing hulls $K_t$ 
($K_t\subset\overline{\mathbb{H}}$, $K_t=\overline{\mathbb{H}}\cap\overline{K_t}$, $\mathbb{H}\setminus K_t$ simply connected, $K_t$ right-continuous, $K_t$ continuously increasing), after possibly reparameterizing we can find a conformal map $g_t:\mathbb{H}\setminus K_t\to\mathbb{H}$ and a continuous function $\lambda:[0,T]\to\mathbb{R}$ satisfying (\ref{eqn:le}) and (\ref{eqn:hydronorm}). This gives a one-to-one correspondence between families of hulls and real-valued continuous functions. For more details, see Section 4.1 in \cite{lawler}.

\begin{figure}
    \centering
    \includegraphics[width=1.3in]{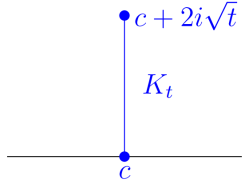}
    \caption{Hull driven by $\lambda(t)\equiv c$.}
    \label{fig:hullofconstantfunction}
\end{figure}

For a basic example, let $\lambda(t)\equiv c\in\mathbb{R}$. Then $g_t(z)=\sqrt{(z-c)^2+4t}+c$ and $K_t=\gamma[0,t]$, where $\gamma(t)=c+2i\sqrt{t}$ is the vertical slit from $c$ to $c+2i\sqrt{t}$ (see Figure \ref{fig:hullofconstantfunction}). This also shows that $\hcap([c,c+2i\sqrt{t}])=t$. 
On the other hand, for a non-constant $\lambda(t)$ we see a non-vertical hull growth that is not as tall as the hull of the constant driving function.

We mention two important properties of the Loewner equation: 
\begin{itemize}
\item Scaling Property:  If $K_t$ is generated by $\lambda(t)$, then for $r>0$,  the scaled hull $\frac{1}{r}K_{r^2t}$ is generated by $\frac{1}{r}\lambda(r^2t)$.
\item Concatenation Property: Let $\l : [0,T] \to \mathbb{R}$ generate $K_t$ and $g_t$, and let $s \in (0,T)$.  Define $\hat{\l}$ on $[s, T]$ to be the restricted function $\l |_{[s,T]}$, and let $\hat{K}$ be the final hull generated by $\hat{\l}$.
Then $K_T = K_s \cup g_s^{-1}( \hat{K} )$.
\end{itemize}

When $\lambda(t)=\sqrt{\kappa}B_t$ where $B_t$ is a Brownian motion starting at 0 and $\kappa>0$, a random family of curves is generated via the Loewner equation, and we call them the \textit{Schramm--Loewner Evolution} (SLE$_\kappa$). For SLE$_\kappa$, it is well-known   that almost surely the trace $\gamma(t)$ exists \cite{basicpropertiesofsle}. Note that the self-similarity of the Brownian motion (i.e. $(rB_t)_{t\ge 0}\stackrel{\textrm{d}}{=}(B_{r^2t})_{t\ge 0}$ for each $r>0$, where $\stackrel{\textrm{d}}{=}$ means equality in distribution) is the same as the scaling of the Loewner equation, which gives the same self-similarity of the hulls and traces (i.e. $(rK_t)_{t\ge 0}\stackrel{\textrm{d}}{=}(K_{r^2t})_{t\ge 0}$ and $(r\gamma(t))_{t\ge 0}\stackrel{\textrm{d}}{=}(\gamma(r^2t))_{t\ge 0}$).  
This self-similarity means that SLE$_\kappa$ is invariant under scaling by a real constant, 
which allows for different geometric behavior for different values of $\kappa$. 
Rohde and Schramm \cite{basicpropertiesofsle} proved that SLE$_\kappa$ exhibits phase transitions based on $\kappa$ as follows: 
\begin{itemize}
    \item $\kappa\in[0,4]:$ $\gamma(t)$ is a.s.\ a simple path in $\mathbb{H}\cup\{0\}$
    \item $\kappa\in(4,8):$ $\gamma(t)$ is a.s.\ a non-simple path
    \item $\kappa\in[8,\infty)$: $\gamma(t)$ is a.s.\ a spacefilling curve
\end{itemize}

In the deterministic setting, a natural class of functions to consider is H\"older continuous functions of exponent $\frac{1}{2}$, also known as Lip$(\frac{1}{2})$.  This is the set of functions such that $|\lambda(s)-\lambda(t)|\leq c\sqrt{|t-s|}$, and the norm, denoted $\|\lambda\|_{1/2}$, is the smallest such $c$. Note that for $r>0$, $r\|\lambda\|_{1/2}=\|\lambda(r^2t)\|_{1/2}$, which is the scaling of the Loewner equation. The phase transitions in this setting are different from SLE$_\kappa$ and slightly more complicated since for any $c>0$ we can find $\lambda\in\text{Lip}(\frac{1}{2})$ so that $\|\lambda\|_{1/2}=c$ and $\lambda$ generates a simple curve. Nevertheless, the following theorem and the example discussed afterwards show that  a
 deterministic phase transition occurs when $\|\lambda\|_{1/2}=4$.
\begin{thm}[\cite{sharpcondition}] \label{thm:cis4thm}
If $\lambda\in\text{Lip}(\frac{1}{2})$ with $\|\lambda\|_{1/2}<4$, then the domains $\mathbb{H}\setminus K_t$ generated by $\lambda$ are quasi-slit halfplanes.  In particular, $K_t$ is a simple curve in $\mathbb{H}$.
\end{thm}

When considering phase transitions, the key deterministic example  is $\lambda(t)=c\sqrt{1-t}$. For $c<4$, 
$\lambda$ generates a simple curve,
whereas if $c\geq 4$, then 
$\lambda$ does not generate a simple curve.
For more details, see \cite{knk}.
The behavior of this family for  $c\geq 4$ was leveraged in \cite{lindrobins} to show that the scaled Weierstrass function generates a non-simple hull for a large enough scale.  (See Section 4.2 for a further discussion of the Weierstrass example.)  We state their technique in the following theorem, which we will use in proving Theorem \ref{thm:timechangesqrtnotsimple}.

\begin{thm}[\cite{lindrobins}]\label{thm:lindrobinsthm}
Let $\lambda(t)$ be a continuous function. If $|\lambda(T)-\lambda(t)|\geq 4\sqrt{T-t}$ on $[T-\epsilon,T]$ for some $\epsilon>0$, then the hull generated by $\lambda$ at time $t=T$ is non-simple.
\end{thm}

\subsection{L\'evy Processes and Random Time Changes}

One of our goals is to consider a Brownian motion that has been time-changed by the inverse of a stable subordinator. We will use the inverse in order for our time change to be continuous. However, knowing how the subordinator behaves will help us see how its inverse behaves. In this subsection, we will set up the necessary definitions and develop some intuition about these processes. 
The discussion begins with the definition of L\'evy processes as they include Brownian motion and stable subordinators as special cases.
Throughout the paper, given stochastic processes are assumed to be defined on a probability space $(\Omega,\mathcal{F},\mathbb{P})$, take values in $\mathbb{R}$, start at 0, and have right-continuous sample paths with left limits.

A stochastic process $X=(X_t)_{t\geq 0}$ is called a \textit{L\'evy process} if it is stochastically continuous (i.e.\ for $\epsilon>0$ and $t\geq 0$, $\lim_{s\to t}\P(|X_{s}-X_t|>\epsilon)=0$) and has stationary and independent increments. 
If $X$ is a L\'evy process, then for each $t>0$, the random variable $X_t$ is infinitely divisible and its distribution is characterized by the triplet $(b,\sigma^2,\nu)$ appearing in the so-called L\'evy-Khintchine formula 
\[
	\mathbb{E}[e^{i uX_t}]=e^{t\eta(u)} \ \  \textrm{with} \ \ \eta(u)=ibu-\frac 12 \sigma^2u^2+\int_{\mathbb{R}\setminus \{0\}}(e^{iuy}-1-iuy\mathbf{1}_{|y|<1})\nu(\textrm{d}y),
\]
where $\mathbb{E}$ denotes the expectation under $\mathbb{P}$. 
 Here, $b\in\mathbb{R}$, $\sigma^2\ge 0$, and $\nu$ is a Borel measure on $\mathbb{R}\setminus \{0\}$ with $\int_{\mathbb{R}\setminus \{0\}}(|y|^2\wedge 1)\nu(\textrm{d}y)<\infty$, called the \textit{L\'evy measure} of $X$. 
Since sample paths of a L\'evy process are assumed to be right-continuous with left limits, they have at most countably many jumps, but for each $\epsilon>0$, they have only finitely many jumps of size $\epsilon$ or larger. The L\'evy measure $\nu$ controls the jumps of the L\'evy process. 
In particular, Brownian motion is a L\'evy process with triplet $(b,\sigma^2,\nu)=(0,1,0)$ so that $\mathbb{E}[e^{i uX_t}]=e^{-\frac 12 u^2t}$. 

A L\'evy process with non-decreasing sample paths is called a \textit{subordinator.} The distribution of a subordinator $D=(D_t)_{t\ge 0}$ is characterized by its Laplace transform
\[
	\mathbb{E}[e^{uD_t}]=e^{-t\psi(u)} \ \ \textrm{with} \ \ \psi(u)=bu+\int_0^\infty (1-e^{-uy})\nu(\textrm{d}y), 
\]
where the L\'evy measure $\nu$ satisfies $\nu(-\infty,0)=0$ and $\int_0^\infty (y\wedge 1) \nu(\textrm{d}y)<\infty$. The function $\psi(u)$ is called the \textit{Laplace exponent} of $D$.  
In this paper, we only consider a subordinator $D$ with $b=0$ and $\nu(0,\infty)=\infty$, which implies that $D$ has strictly increasing sample paths with $\lim_{t\to\infty}D_t=\infty$ and the jump times of $D$ are dense in $(0,\infty)$ (see \cite{Sato}). 
  Examples of theoretically and practically important subordinators include: 
\begin{itemize}
\item a \textit{stable subordinator} of index $\alpha\in(0,1)$ (or an $\alpha$-stable subordinator for short), where $\nu(\textrm{d}x)=\frac{\alpha}{\Gamma(1-\alpha)}x^{-\alpha-1}\mathbf{1}_{x>0}\,\textrm{d}x$ and $\psi(u)=u^\alpha$, and 
\item a \textit{tempered stable subordinator} of index $\alpha\in(0,1)$ and tempering factor $\theta>0$, where $\nu(\textrm{d}x)=\frac{\alpha}{\Gamma(1-\alpha)}e^{-\theta x}x^{-\alpha-1}\mathbf{1}_{x>0}\,\textrm{d}x$ and $\psi(u)=(u+\theta)^\alpha-\theta^\alpha$. 
\end{itemize}
Tempered stable subordinators may be regarded as a one-parameter extension of stable subordinators via $\theta$. However, they possess very different properties. Indeed, while a stable subordinator has infinite first moment, a tempered stable subordinator has finite moments of all orders due to the factor $e^{-\theta x}$ which diminishes (or ``tempers'') large jumps of the stable subordinator of the same index (see \cite{Rosinski_tempering} for a detailed account of more general tempered stable L\'evy processes). On the other hand, an $\alpha$-stable subordinator is the only subordinator which is self-similar with index $1/\alpha$; i.e.\ $(D_{ct})\stackrel{\textrm{d}}{=}(c^{1/\alpha} D_t)$ for all $c>0$ (see \cite{EmbrechtsMaejima}).

The L\'evy--It\^o decomposition allows us to develop some intuition about the jumps of subordinators, which will in turn help us view their inverses. 
Given a L\'evy process $X$, for each Borel set $A$ of $\mathbb{R}\setminus\{0\}$ and $t\geq 0$, define
\begin{equation}
    N(t,A)(\omega)=\#\{0\leq s\leq t:\Delta X_s(\omega)\in A\}, 
\end{equation}
where $\Delta X_s$ is the size of the jump at $s$.  
For fixed $t>0$ and $\omega\in\Omega$, $N(t,\cdot)(\omega)$ is a counting measure on the collection of Borel sets of $\mathbb{R}\setminus\{0\}$. For $A$ bounded below, $(N(t,A))_{t\geq 0}$ is a Poisson process with intensity $\mu(A)=\mathbb{E}(N(1,A))$. 
$N(t,A)$ is called the Poisson random measure associated with $X$. 
For $A$ bounded below, $t>0$ and $\omega\in\Omega$, define the Poisson integral of $x$ with respect to the random measure as the random finite sum
\begin{equation}
    \int_A xN(t,\textrm{d}x)(\omega)=\sum_{x\in A}xN(t,\{x\})(\omega)
    =\sum_{0\le s\le t}\Delta X_s(\omega) \mathbf{1}_A(\Delta X_s(\omega)).
\end{equation}
Define $\widetilde{N}$, the compensated Poisson random measure of $X$, by $\widetilde{N}(t,A)=N(t,A)-\mathbb{E}(N(t,A))=N(t,A)-t\mu(A)$. 
The L\'evy--It\^o decomposition states that any L\'evy process $X$ can be expressed as 
\begin{equation}
    X_t=b_1 t+\sigma B_t+\int_{|x|<1}x\widetilde{N}(t,\textrm{d}x)+\int_{|x|\geq 1}xN(t,\textrm{d}x),
\end{equation}
where $b_1\in\mathbb{R}$ and $\sigma B_t$ is a scaled Brownian motion independent of the Poisson random measure $N$. 
Note that the terms $\int_{|x|<1}x\widetilde{N}(t,\textrm{d}x)$ and $\int_{|x|\geq 1}xN(t,\textrm{d}x)$ represent small jumps and large jumps of $X$, respectively. 
It also follows that $X$ has finite variation if and only if its L\'evy--It\^o decomposition can be rewritten as 
\begin{equation}
    X_t=\left(b_1-\int_{|x|<1}x\nu(\textrm{d}x)\right)t+\int_{\mathbb{R}\setminus \{0\}} xN(t,\textrm{d}x). 
\end{equation}

\begin{figure}
    \centering
    \begin{minipage}{2in}
    \includegraphics[width=2in]{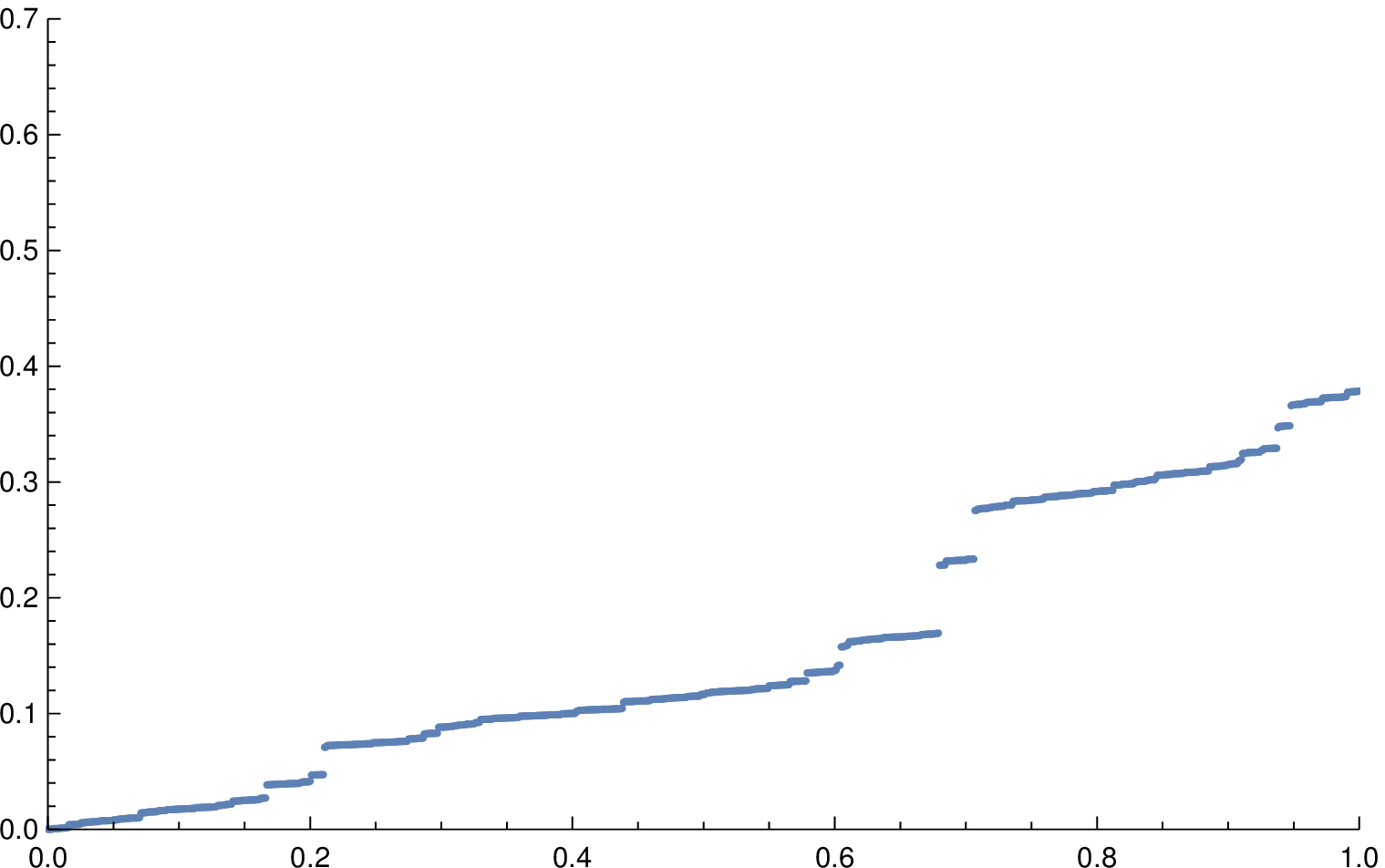}
    \end{minipage}$\qquad$
    \begin{minipage}{2in}
    \includegraphics[width=2in]{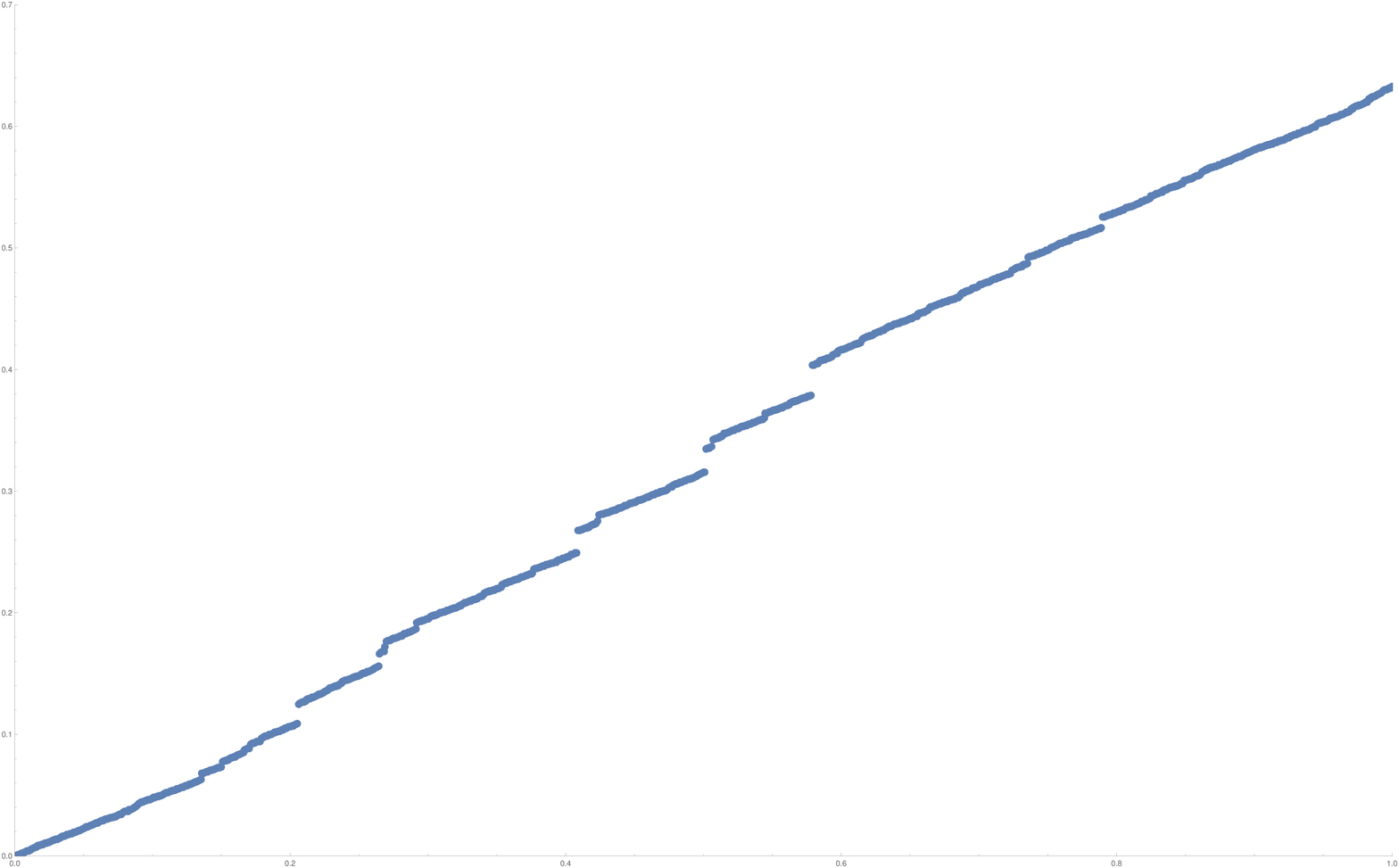}
    \end{minipage}
\caption{Sample paths of a stable subordinator with $\alpha=0.7$ (left) and $\alpha=0.9$ (right).}
\label{fig:subordinatorsamplepaths}
\end{figure}

An $\alpha$-stable subordinator $D$, which is strictly increasing (and hence of finite variation), can be expressed as 
\begin{equation}
    D_t=\lim_{n\to\infty}\left(c_n t+\int_{x\geq\epsilon_n}xN(t,\textrm{d}x)\right)=\int_{x>0}xN(t,\textrm{d}x),
\end{equation}
where $\epsilon_n\downarrow 0$ and $c_n=\E[\int_{0<|x|<\epsilon_n}xN(1,\textrm{d}x)]$. 
Hence, we can think of the sample paths of $D$ as approximated by a process that has finitely many jumps, where the jump sizes are bounded below, and between jumps it is linear, see Figure \ref{fig:subordinatorsamplepaths}. On the other hand, the sample paths of $D$ itself can increase only by jumps. 
This approximation argument comes from the idea of ``interlacing,'' the details of which appear in Section 2.6.2 of \cite{Applebaum}.

Define the \textit{inverse} (or the \textit{first hitting time process}) $E=(E_t)_{t\ge 0}$ of a subordinator $D$ by
\begin{equation}
    E_t=\inf\{s>0: D_s>t\}.
\end{equation}
We refer to $E$ as an \textit{inverse subordinator} for short. With the assumption that $b=0$ and $\nu(0,\infty)=\infty$, $D$ has strictly increasing paths starting at 0, and hence, its inverse $E$ has continuous, non-decreasing paths starting at 0 which are not constant in a neighborhood of $t=0$. 
Indeed, we can think of $E$ as having random long flat periods (corresponding to the large jumps of $D$) and in between these, $E$ is increasing very quickly (since $D$ has infinitely many small jumps). Note that $E$ is not a L\'evy process since it no longer has independent or stationary increments (see \cite{MS-1}). 
On the other hand, $E$ has finite exponential moment; i.e.\ $\mathbb{E}[e^{cE_t}]<\infty$ for all $t$ (see e.g.\  \cite{inversesubordinatorsimulation}).

\begin{figure}
    \centering
    \begin{minipage}{2in}
    \includegraphics[width=2in]{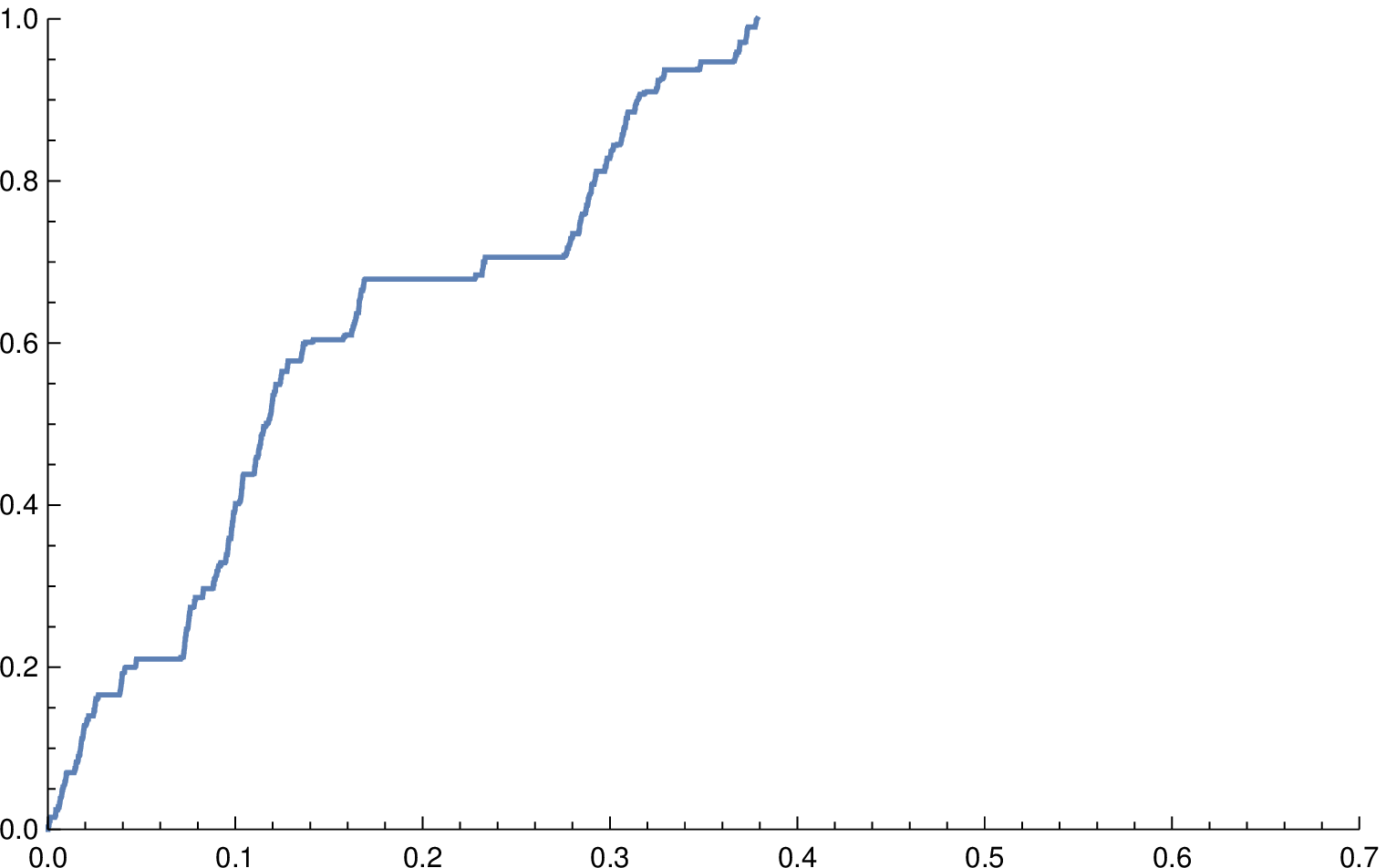}
    \end{minipage}$\qquad$
    \begin{minipage}{2in}
    \includegraphics[width=2in]{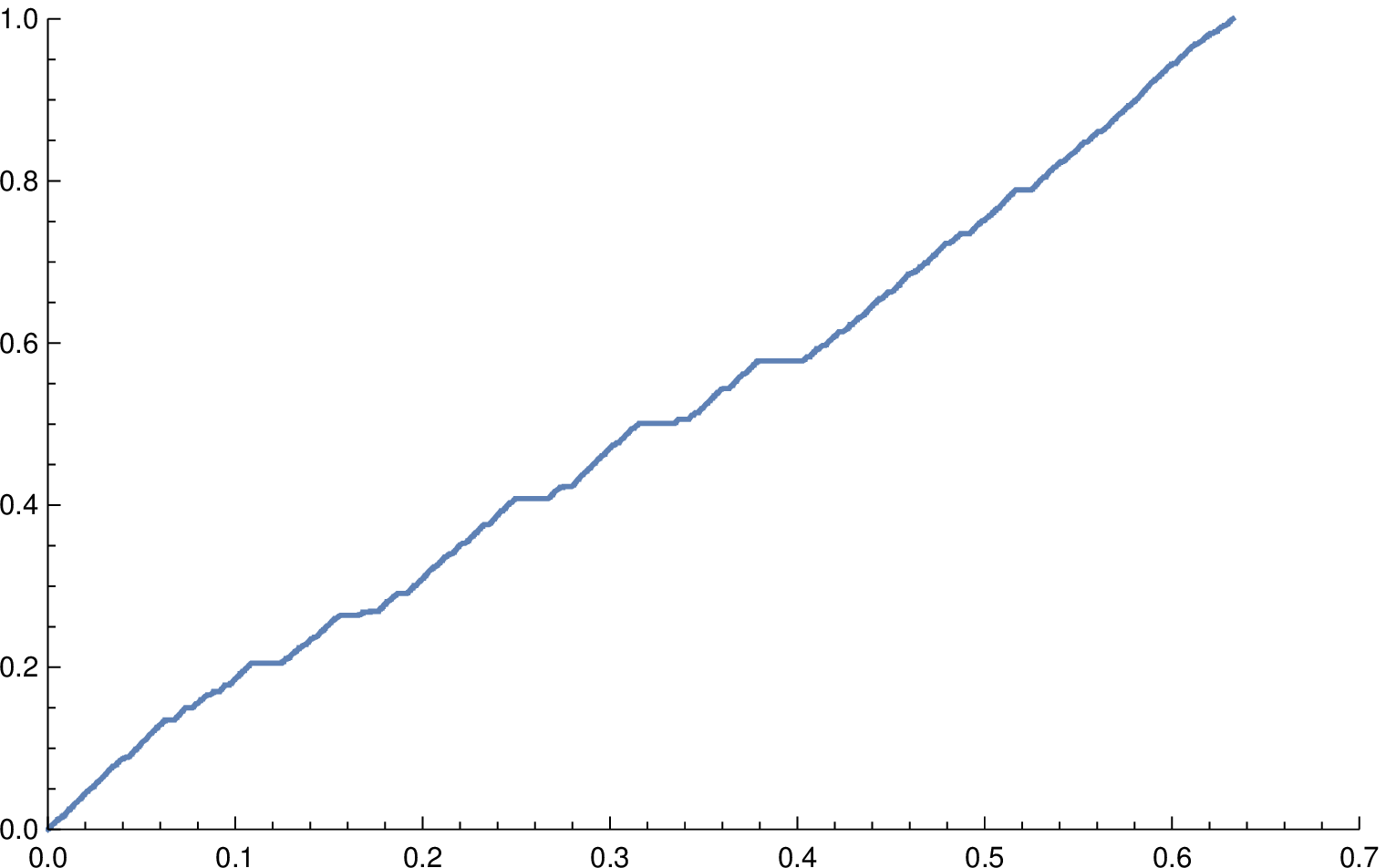}
    \end{minipage}
   
    \vspace{0.2in}
  
    \begin{minipage}{2in}
    \includegraphics[width=2in]{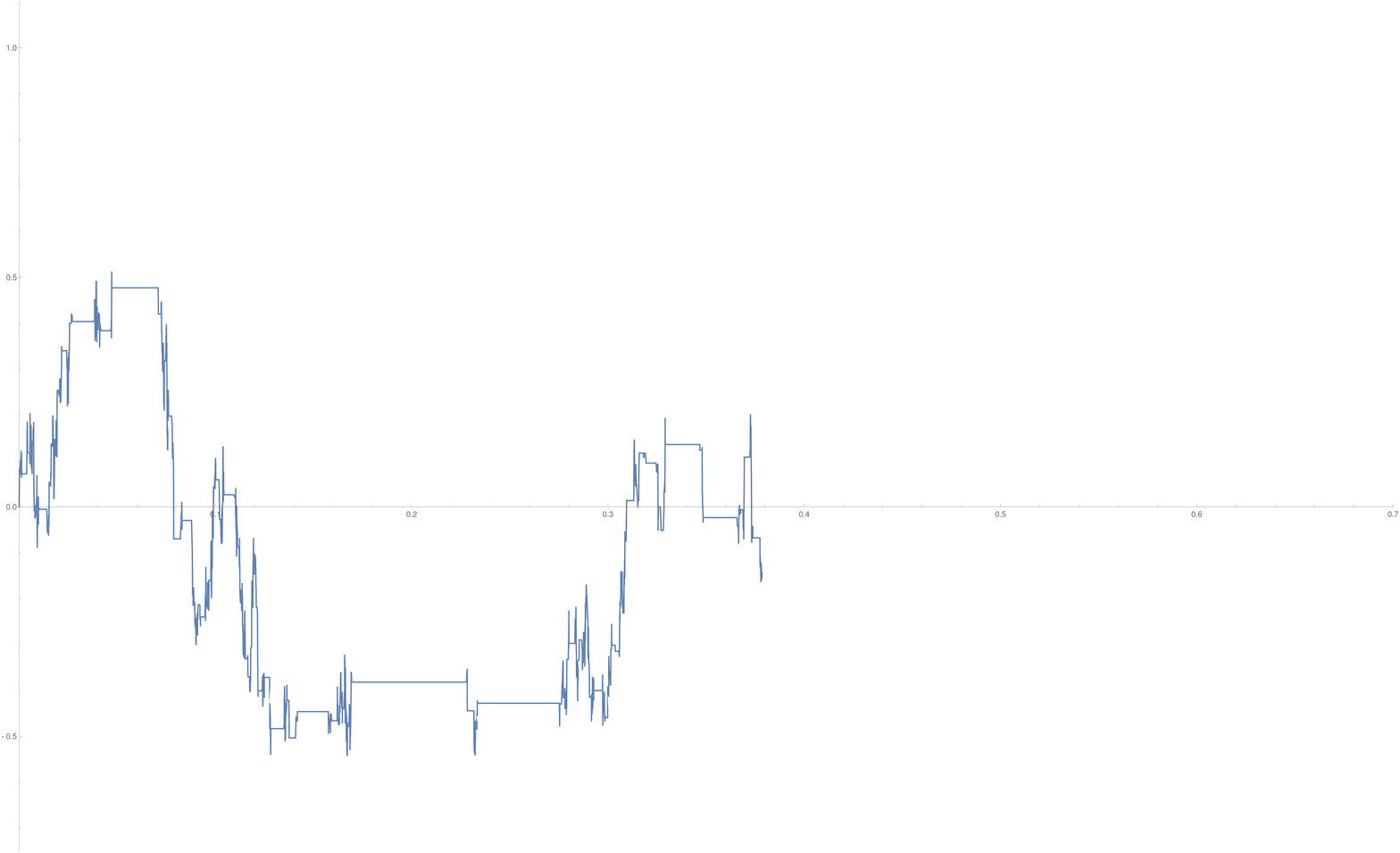}
    \end{minipage}$\qquad$
    \begin{minipage}{2in}
    \includegraphics[width=2in]{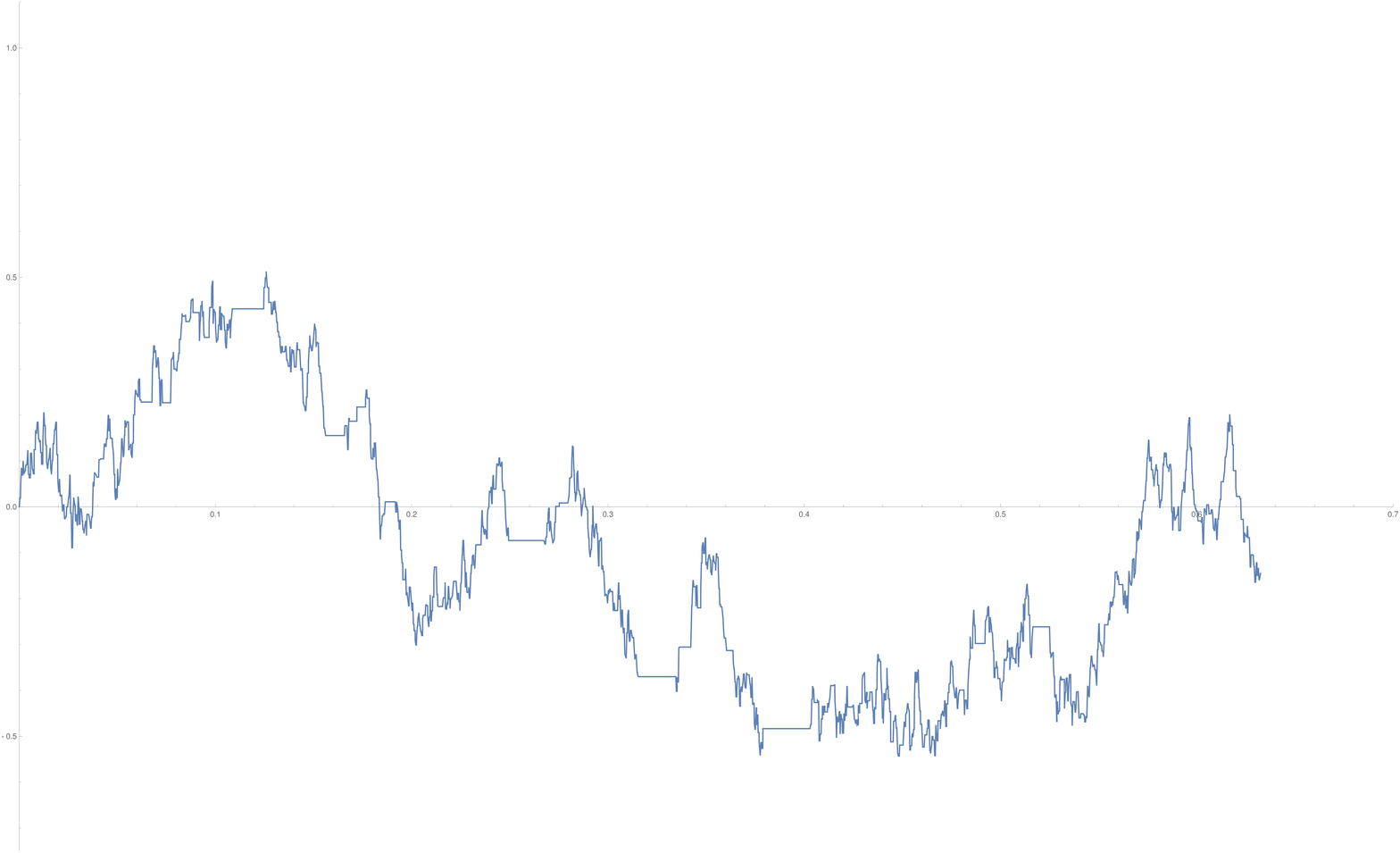}
    \end{minipage}
\caption{Sample paths of an inverse stable subordinator $E_t$ with $\alpha=0.7$ (top left) and $\alpha=0.9$ (top right),
and the corresponding sample paths of the time-changed Brownian motion $B_{E_t}$ with $\alpha=0.7$ (bottom left) and $\alpha=0.9$ (bottom right).}
\label{fig:inversesamplepaths}
\end{figure}

Now, suppose $D$ is an $\alpha$-stable subordinator independent of Brownian motion $B$. Then the self-similarity of $D$ with index $1/\alpha$ implies self-similarity of $E$ with index $\alpha$ (See \cite{MS-1}). 
Figure \ref{fig:inversesamplepaths} presents sample paths of the time change $E$ and the corresponding \textit{time-changed Brownian motion} $B\circ E=(B_{E_t})_{t\ge 0}$. 
The time-changed Brownian motion is non-Markovian and non-Gaussian (\cite{MS-1,MS-2}). Moreover, the densities $p(t,x)$ of $B_{E_t}$ satisfy the time-fractional order heat equation 
\[
	\frac{\partial^\alpha p(t,x)}{\partial t^\alpha}=\frac 12\frac{\partial^2 p(t,x)}{\partial x^2},
\]
 where $\partial^\alpha/\partial t^\alpha$ is the Caputo fractional derivative of order $\alpha$ (see e.g.\ \cite{Gorenflo1997}). The time-changed Brownian motion and stochastic differential equations it drives have been used to model subdiffusions, where particles spread at a slower rate than the usual Brownian particles (see e.g.\ \cite{MeerschaertSikorskii,meerschaert2009,MagdziarzSchilling,Hahn2012,Kobayashi2011} and references therein). 
  For simulations of inverse subordinators and their associated time-changed processes, see the algorithms presented and discussed in e.g.\ \cite{GadjaMagdziarz,inversesubordinatorsimulation}.


\section{Limiting Hull Behavior}\label{scalingarguments}

\subsection{Rescaled Hulls}

This section develops results about rescaled hulls 
for time-changed processes of the form $X\circ E=(X_{E_t})_{t\ge 0}$, where $X=(X_t)_{t\ge 0}$ and $E=(E_t)_{t\ge 0}$ are independent self-similar processes with continuous paths. 
Note that $E$ is not necessarily non-decreasing and may be allowed to take negative values if the process $(X_t)$ is defined for $t\in\mathbb{R}$; for simplicity of discussion, however, we assume that $E$ is nonnegative. 
Important examples of the ``outer process'' $X$ include fractional Brownian motion $B^H$ of Hurst index $H\in(0,1)$, which coincides with Brownian motion when $H=1/2$. When $H\ne 1/2$, $B^H$ is non-Markovian and its increments are positively correlated if $H>1/2$ and negatively correlated if $H<1/2$. 
On the other hand, an inverse $\alpha$-stable subordinator can serve as the ``inner process'' $E$ with self-similarity index $\alpha\in(0,1)$.  Since we are allowed to take $E$ to be the identity map, the results presented in this section cover the cases of (untime-changed) self-similar processes as well. 

Recall that a process $(X_t)$ is said to be self-similar with index $H>0$ if $(X_{ct})\stackrel{\textrm{d}}{=}(c^H X_t)$ for all $c>0$. Note that the self-similarity implies that $X_0=0$ a.s. We begin with a simple lemma.

\begin{lemma}\label{lem:scalingoftimechangedhulls_general}
Let $K_t$ be the hull driven by $\lambda(t)=X_{E_t}$, where $X$ is a continuous, self-similar process of index $H>0$ and $E$ is a nonnegative, continuous, self-similar process of index $\alpha>0$, independent of $X$. 
Then for any $r>0$, the scaled hulls $(\frac{1}{r}K_{r^2t})_{t\geq 0}$ and the hulls driven by $(r^{2H\alpha-1}X_{E_t})_{t\geq 0}$ have the same distribution.
\end{lemma}

\begin{proof}
For any fixed $r>0$, due to the scaling of the Loewner equation, the hulls $\frac{1}{r}K_{r^2t}$ are driven by $\frac{1}{r}X_{E_{r^2t}}$.  On the other hand, the self-similarities of $X$ and $E$ together with independence imply that $(X_{E_t})$ is self-similar with index $H\alpha$, so 
\[
	\Bigl(\frac{1}{r}X_{E_{r^2t}}\Bigr)_{t\ge 0}
	\stackrel{\textrm{d}}{=} \Bigl(\frac{1}{r} (r^2)^{H\alpha} X_{E_t}\Bigr)_{t\ge 0}
	=(r^{2H\alpha-1} X_{E_t})_{t\ge 0}.
\]
Therefore, the hulls generated by $(r^{2H\alpha-1} X_{E_t})_{t\ge 0}$ must have the same distribution as the scaled hulls $(\frac{1}{r}K_{r^2t})_{t\ge 0}$. 
\end{proof}

In \cite{chenrohde}, Chen and Rohde consider geometric properties of the Loewner hulls that are generated by a symmetric stable process.   
One of their results (Proposition 3.2 in that paper) shows that rescaling the hulls  leads to deterministic sets 
(and fairly uninteresting sets -- either a vertical line segment or the empty set). 
This is expected because the driving process does not satisfy Brownian scaling, which implies that the Loewner hulls will not satisfy scale-invariance.

The following result, which is analogous to Proposition 3.2 in \cite{chenrohde}, holds for our time-changed process $(X_{E_t})$. 

\begin{prop}\label{thm:rescaledhullslines}
Let $K_t$ be the hull  driven by $X_{E_t}$, where $X$ is a continuous, self-similar process of index $H>0$ and $E$ is a nonnegative, continuous, self-similar process of index $\alpha>0$, independent of $X$. 
\begin{enumerate}
\item[(a)] 
If $2H\alpha<1$, then as $r \to \infty$, the rescaled hulls $\frac{1}{r}K_{r^2}$ converge to the vertical line segment $[0,2i]$ (in the Hausdorff metric) in probability. If $2H\alpha>1$, then the same conclusion holds as $r \to 0$.
\item[(b)] Suppose $\P(X_1=0)=0$. 
If $2H\alpha<1$, for all $\epsilon >0$, 
\[
	\displaystyle \lim_{r \to 0} \P \left( \frac{1}{r} K_{r^2} \cap \{ y > \epsilon \text{ and } |x| < 1/\epsilon \} \neq \emptyset \right) = 0.
\]
 If $2H\alpha>1$, then the same conclusion holds with the limit as $r \to \infty$. 
\end{enumerate}
\end{prop}

\begin{remark}
\begin{em}
In part (b), we need the constraint $|x|<1/\epsilon$ on the real part, which can be interpreted as follows. If $2H\alpha<1$ and $r$ is very small, then $r^{2H\alpha-1} \lambda(t)$ has a very large scaling factor. Near zero, we would not expect to see much hull growth, but far away from zero, we will see the taller parts of the hull that are grown when the driving function is constant. 
\end{em}
\end{remark}

The proof utilizes the next lemma, which is entirely deterministic:

\begin{lemma}[Lemma 3.3 in \cite{chenrohde}]\label{lem:lemma33}
(a) If $\lambda(t)\in[a,b]$ for all $t\in[0,T]$, then $K_T\subset [a,b]\times\mathbb{R}$.

(b) Let $0<\epsilon<1$. If $I\subset\mathbb{R}$ is an interval of length $\sqrt{T}$ and $10I$ the concentric interval of size $10\sqrt{T}$, and if
\begin{equation}
    \int_{0}^{T}\mathbf{1}_{\{\lambda(t)\in 10I\}}\mathrm{d}t\leq\epsilon T,
\end{equation}
then
\begin{equation}
    K_T\cap I\times[4\sqrt{\epsilon T},\infty)=\emptyset.
\end{equation}
\end{lemma}

\begin{proof}[Proof of Proposition \ref{thm:rescaledhullslines}] 
To prove (a), note that due to Lemma \ref{lem:scalingoftimechangedhulls_general}, the scaled hull $\frac{1}{r}K_{r^2}$ is equal in distribution to the time $t=1$ hull generated by 
$\lambda_r(t) = r^{2H\alpha -1}\lambda(t)$, where $\lambda(t)=X_{E_t}$. 
Since $(X_{E_t})$ has continuous paths, we have $\sup_{0 \leq t \leq 1} |\lambda(t)|<\infty$ a.s.
Therefore, if $2H\alpha<1$, for any $\epsilon>0$, 
\[
\lim_{r\to\infty}\P\left( \exists \,t \in [0,1] \text{ such that } |\lambda_r(t)| > \epsilon \right)
=\lim_{r\to\infty}\P \left( \sup_{0 \leq t \leq 1} |\lambda(t)| > r^{1-2H\alpha}\epsilon \right)
=0.
\]
Hence
\[
 \lim_{r\to\infty}\P\left( \l_r[0,1] \subset [-\epsilon, \epsilon] \right) = 1.
 \]
By part (a) of Lemma \ref{lem:lemma33}, this implies that the width of the time 1 hull of $\lambda_r(t)$ is going to 0 in probability, so the same is true for $\frac{1}{r}K_{r^2}$.
Since the halfplane capacity of $\frac{1}{r}K_{r^2}$ is 1, we know that it must be converging to the interval $[0,2i]$ with respect to Hausdorff distance (see the example with $\lambda(t)\equiv c$ in Section \ref{background}). The case when $2H\alpha>1$ is proved in an analogous manner. 

To prove (b), let $\epsilon\in (0,2)$. 
Once again, by Lemma \ref{lem:scalingoftimechangedhulls_general}, the scaled hull $\frac{1}{r}K_{r^2}$ is equal in distribution to the time $t=1$ hull generated by 
$r^{2H\alpha -1}\lambda(t)$ (call this hull $\widetilde{K}_1$). So
\[
\P \left( \frac{1}{r} K_{r^2} \cap \{ y > \epsilon \text{ and } |x| < 1/\epsilon \} \neq \emptyset \right)
=\P \left( \widetilde{K}_1 \cap \{ y > \epsilon \text{ and } |x| < 1/\epsilon \} \neq \emptyset \right).
\]
Now, if $\widetilde{K}_1 \cap \{ y > \epsilon \text{ and } |x| < 1/\epsilon \} \neq \emptyset$, then there exists a (random) interval $I$ of length 1 (not necessarily centered at 0) such that $I\subset(-1/{\epsilon},1/{\epsilon})$ and $\widetilde{K}_1 \cap I\times (\epsilon,\infty) \neq \emptyset$. 
By Lemma \ref{lem:lemma33}(b) with $T=1$ and $\epsilon$ replaced by $4\sqrt{\epsilon}$, this implies $\int_0^1\1_{\{r^{2H\alpha-1}\lambda(t)\in 10I\}}\mathrm{d}t>\epsilon^2/16$. 
 Since 
$\{ r^{2H\alpha-1}\lambda(t)\in 10I \} \subset \{ |\lambda(t)|\leq\delta/\epsilon \}$
 with $\delta=10r^{1-2H\alpha}$, it follows that 
\begin{align}
\P \left( \frac{1}{r} K_{r^2} \cap \{ y > \epsilon \text{ and } |x| < 1/\epsilon \} \neq \emptyset \right)
&\leq\P\left(\int_{0}^{1}\1_{\{|\lambda(t)|\leq\delta/\epsilon\}}\mathrm{d}t>\frac{\epsilon^2}{16}\right)\\
&\le\frac{16}{\epsilon^2}\int_0^1 \P\left(|\lambda(t)|\le \frac{\delta}{\epsilon}\right)\,\mathrm{d}t,
\end{align}
using Markov's inequality.
Since $\P(X_1=0)=0$ and the process $X$ is self-similar by assumption,
$\P(X_t=0)=0$ for each $t>0$. Moreover, since $X$ and $E$ are independent, $\P(|\lambda(t)|=0)=\P(X_{E_t}=0)=\E[\P(X_{E_t}=0|E_t)]=0$ for each $t>0$. Therefore, the above bound tends to 0 upon taking the limit as $\delta\to0$ (so as $r\to 0$ if $2H\alpha<1$ or as $r\to\infty$ if $2H\alpha>1$), which completes the proof. 
\end{proof}

Proposition \ref{thm:rescaledhullslines} immediately yields the following corollary:

\begin{cor}\label{cor:rescaledhullslines}
Let $B$ be a Brownian motion independent of an inverse $\alpha$-stable subordinator $E$. Let $K_t$ be the hull driven by the time-changed Brownian motion $B_{E_t}$.
\begin{enumerate}
\item[(a)]
As $r \to \infty$, the rescaled hulls $\frac{1}{r}K_{r^2}$ converge to the vertical line segment $[0,2i]$ (in the Hausdorff metric) in probability. 
\item[(b)] For all $\epsilon >0$, 
\[
	\displaystyle \lim_{r \to 0} \P \left( \frac{1}{r} K_{r^2} \cap \{ y > \epsilon \text{ and } |x| < 1/\epsilon \} \neq \emptyset \right) = 0.
\]
\end{enumerate}
\end{cor}

\subsection{Tangential Hulls}

Proposition \ref{thm:rescaledhullslines}(b), in the case that $2H\alpha<1$, tells us that the initial hull growth stays close to the real line.  In some cases, we can describe tangential hull behavior more concretely. 
In particular, 
the following deterministic result tells us that if the driving function is moving faster than a square root function at $t=0$, then the Loewner hull leaves the real line tangentially. 
The result will be used in Section \ref{deterministicexample}.

\begin{thm:tangentialdeparture}
Suppose that $\l(0)=0$ and $\l(t) \geq a t^r$  where $a>0$ and $r \in (0, 1/2]$.  
Then for $t$ small enough, the  hull $K_t$ driven by $\l$ is contained in the region 
$\{ x+iy \, : \, 0\leq x, \, 0<y< \frac{26}{a} \,x^{2-2r} \}$.
\end{thm:tangentialdeparture}

The proof of this proposition will follow from scaling and the next lemma.
 
\begin{lemma}\label{fromCRlem}
Let  $k > 0$.
Suppose $\l$  is defined on $[0,T]$ for $T \geq 1$ and satisfies that 
$\lambda(0)=0$, $\lambda(t) \geq k \sqrt{t}$ for $t \in [0,1]$, and $\l(t) > 3.5$ for $t\geq1$.
Let $K_t$ be driven by $\l$.
Then $\displaystyle K_t \cap \left( [1,2] \times [26/k, \infty) \right) = \emptyset$ for all $t$.
\end{lemma}

\begin{proof}
First assume that $k > 6.5$.  Then for 
 $t \in [0,1]$, the amount of time that $\l$ spends  in $[-3.5, 6.5]$ is at most  $(6.5/k)^2 < 1$.  
Therefore, applying Lemma \ref{lem:lemma33}(b) with $\epsilon = (6.5/k)^2$, 
we conclude that $K_1$ does not intersect $[1,2] \times [26/k, \infty)$.
If $0< k \leq 6.5$, it is trivially true that $K_1 \cap  \left( [1,2] \times [26/k, \infty) \right) = \emptyset $, since the maximum height of $K_1$ is 2.
Let $z \in [1,2] \times [26/k, \infty)$.
The proof of Lemma \ref{lem:lemma33}(b) further gives that 
Re$(g_1(z)) \leq 3.5$
for the $k>6.5$ case, and one can check that this still holds in the $0< k \leq 6.5$ case, too.
Note that for $t \geq 1$, $\l(t) >3.5.$
Thus applying Lemma  \ref{lem:lemma33}(a) to $\l$ on $[1,\infty)$,
we see that $g_1(z)$ cannot be part of the hull $g_1(K_t\setminus K_1)$.
Thus $z$ is not in the hull $K_t$ for any $t$.
\end{proof}

\begin{proof}[Proof of Proposition \ref{tangentialdeparture}]
Let $\l_n(t)  := 2^n\l(2^{-2n}t)$ generate the hull $K_t^n$, 
and note that $\l_n(t) \geq a 2^{n(1-2r)}  t^r.$
Choose $N$ so that $a 2^{N(1-2r)}  \geq 3.5$ and $\l_N$ is defined on $[0,1]$. 
For $n \geq N$,  Lemma \ref{fromCRlem} applied to $\l_{n}(t)$ implies
$K_t^{n}$ does not intersect $[1,2] \times [(26/a) \cdot 2^{-n(1-2r)}, \infty)$.
Thus by scaling,  $K_t$ does not intersect $[2^{-n},2^{-n+1}] \times [(26/a) \cdot 2^{-n(2-2r)}, \infty)$ for $n \geq N$. 
Therefore for small $t$, the hull lies below the curve $y=(26/a) \,x^{2-2r}$.
\end{proof}


\section{Hulls Generated with Time Change}\label{notsimple}

In this section, we take a look at the geometric behavior of a hull that has been generated by a time-changed driving function, where the random time change is given by an inverse subordinator.  In particular, we will consider two cases for the driving function, a time-changed deterministic function and a time-changed Brownian motion; and we ask whether or not the generated hulls are simple curves.
Theorem \ref{thm:timechangesqrtnotsimple} gives conditions for determining simpleness or non-simpleness.  This result is applied to the time-changed Brownian case to show that the generated hulls are almost surely non-simple curves.  We end by discussing some examples of time-changed deterministic functions.


\subsection{Criteria for Simple and Non-simple Hulls}

The results to be presented in this section are applicable to a large class of random time changes, including the inverses of stable and tempered stable subordinators. To discuss them, we first introduce the notion of regular variation. A function $\ell:(0,\infty)\to (0,\infty)$ is said to be \textit{slowly varying at $\infty$} if for any $c>0$,  
\[
	\lim_{u\to \infty} \frac{\ell(cu)}{\ell(u)}=1. 
\]
Examples of slowly varying functions include $\ell(u)=(\log u)^\eta$ for any $\eta \in \mathbb{R}$. 
A function $f:(0,\infty)\to (0,\infty)$ is said to be \textit{regularly varying at $\infty$ with index $\alpha>0$} if for any $c>0$, 
\[
	\lim_{u\to \infty} \frac{f(cu)}{f(u)}=c^\alpha. 
\]
Every regularly varying function $f$ with index $\alpha>0$ is represented as $f(u)=u^\alpha \ell(u)$ with $\ell$ being a slowly varying function. For a general account of this topic, consult \cite{regularvariation}. 

Recall that the Laplace transform of a subordinator $D=(D_t)_{t\ge 0}$ with L\'evy measure $\nu$ and zero drift is given by
\[
	\mathbb{E}[e^{-uD_t}]=e^{-t\psi(u)}, \ \ \textrm{where} \ \ 
	\psi(u)=\int_0^\infty (1-e^{-ux})\nu(\textrm{d}x). 
\]
As usual, we assume that sample paths of $D_t$ start at 0 and are right-continuous with left limits and that $\nu(0,\infty)=\infty$ (so that the inverse $E_t$ is continuous). 
In this section, we further assume that the Laplace exponent $\psi$ is regularly varying at $\infty$ with index $\alpha\in(0,1)$. This includes the two important examples of a subordinator: 
\begin{itemize}
\item a stable subordinator with index $\alpha\in(0,1)$, where $\psi(u)=u^\alpha$, and  
\item a tempered stable subordinator with index $\alpha\in(0,1)$ and tempering factor $\theta>0$, where $\psi(u)=(u+\theta)^\alpha-\theta^\alpha$.
\end{itemize}

We now state the main theorem of this section.  Note that the process $X$ in this theorem can be deterministic.

\begin{thm}\label{thm:timechangesqrtnotsimple}
Let $E$ be the inverse of a subordinator $D$ whose L\'evy measure is infinite and Laplace exponent $\psi$ is regularly varying at $\infty$ with index $\alpha\in(0,1)$. 
Let $X$ be a stochastic process with continuous paths. 
\begin{enumerate}
\item[(a)] If $X$ is a.s.\ locally $\beta$-H\"older with $\beta > \frac{1}{2\alpha}$, 
		then a.s.\ $\lambda(t) = X_{E_t}$ generates a simple curve.
\item[(b)] If there exist random variables $\tau, \epsilon, c >0$ and $0<\beta < \frac{1}{2 \alpha}$ so that $\tau$ is independent of the subordinator $D$ and a.s.\ $|X_\tau - X_t| \geq c (\tau - t)^\beta$ for all $t \in (\tau-\epsilon, \tau)$,
		then a.s.\ $\lambda(t) = X_{E_t}$ does not generate a simple curve.
\end{enumerate}
\end{thm}

The proof relies on previously known deterministic results 
(Theorems \ref{thm:cis4thm} and \ref{thm:lindrobinsthm}) 
and the following two lemmas which give control on the local behavior of the time change.
The main idea for proving the second statement
is to compare $\lambda$ with an appropriate square root function that is known to generate a non-simple curve.
This is illustrated in Figure \ref{fig:sqrt} in the case when $X_t$ is the deterministic function $\sqrt{1-t}$.

\begin{lemma}\label{lemma:subbehaviorat0}
Let $D$ be a subordinator whose L\'evy measure is infinite and Laplace exponent $\psi$ is regularly varying at $\infty$ with index $\alpha\in(0,1)$.
Then for any $\gamma\in[1,1/\alpha)$, $\lim_{t\downarrow 0}{D_t}/t^\gamma=0$ almost surely. 
\end{lemma}

\begin{proof}
Fix $\gamma\in[1,1/\alpha)$ and note that the function $h(t)/t$, where $h(t)=2 t^\gamma$, is positive, continuous, and non-decreasing on $(0,\infty)$. By Proposition 47.17 of \cite{Sato} (originally by Fristedt \cite{fristedt}), the desired result follows once we establish $\int_0 \nu[h(t),\infty)\,\textrm{d}t<\infty$. 
However, due to the relation (see the discussion following Chapter III, Proposition 1 in \cite{bertoin}) 
\[
	\psi(1/x)\sim \Gamma(1-\alpha)\nu(x,\infty) \ \ \textrm{as} \ x\downarrow 0,
\]
it suffices to prove that $\int_0 \psi(1/t^\gamma)\textrm{d}t<\infty.$
Upon writing $\psi(u)$ as $\psi(u)=u^\alpha \ell(u)$ using a slowly varying function $\ell$, we obtain 
\[
	\int_0 \psi\Bigl(\frac{1}{t^\gamma}\Bigr)\,\textrm{d}t
	=\dfrac{1}{\gamma}\int^\infty x^{\alpha-1-1/\gamma} \ell(x)\,\textrm{d}x.
\]
Since $\alpha-1-1/\gamma<-1$, the latter integral converges due to Proposition 1.5.10 of \cite{regularvariation}. 
\end{proof}

\begin{figure}
    \centering
    \includegraphics[width=3in]{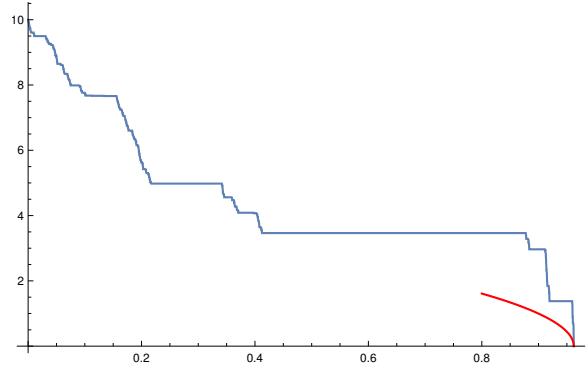}
    \caption{The driving function $\sqrt{1-E_t}$ (blue) is decreasing faster than $4\sqrt{T-t}$ (red) near $T$.}
    \label{fig:sqrt}
\end{figure}

\begin{lemma}\label{lemma:timechangeholder}
Let $D$ be a subordinator whose L\'evy measure is infinite and Laplace exponent $\psi$ is regularly varying at $\infty$ with index $\alpha\in(0,1)$.
Then the inverse $E$ of $D$ is a.s.\ locally weak $\alpha$-H\"older. 
\end{lemma}

\begin{proof}
By Chapter III, Lemma 17 in \cite{bertoin}, an inverse subordinator $E_t$ is a.s.\ locally H\"older with any exponent $\rho\in (0,\underline{\textrm{ind}}(\psi))$, where $\underline{\textrm{ind}}(\psi)$ is the \textit{lower index} of $\psi$ defined as
\[
	\underline{\textrm{ind}}(\psi)=\sup \left\{ \rho>0: \lim_{u\to\infty} \frac{\psi(u)}{u^\rho}=\infty \right\}.
\]
Writing $\psi(u)$ as $\psi(u)=u^\alpha \ell(u)$ and applying Proposition 1.3.6(v) of \cite{regularvariation} yields
\[
	\lim_{u\to\infty} \frac{\psi(u)}{u^\rho}
	=\lim_{u\to\infty} u^{\alpha-\rho}\ell(u)
	=\begin{cases}
		0 &\textrm{if} \ \rho>\alpha,\\
		\infty &\textrm{if} \ 0<\rho<\alpha,
	\end{cases}
\]
which implies $\underline{\textrm{ind}}(\psi)=\alpha$, thereby completing the proof.  
\end{proof}

\begin{proof}[Proof of Theorem \ref{thm:timechangesqrtnotsimple}]

To prove (a), assume that $X$ is a.s.\ locally $\beta$-H\"older with $\beta > \frac{1}{2\alpha}$.  
Since $E_t$ is  a.s.~locally weak $\alpha$-H\"older 
due to Lemma \ref{lemma:timechangeholder},
it follows that $\lambda(t) = X_{E_t}$ is  a.s.\ locally weak $\alpha\beta$-H\"older, where $\alpha\beta > 1/2.$
Thus, for almost every $\omega\in\Omega$, we can partition the time interval $[0,T]$ into a finite number of intervals $J_i:=[t_{i-1}, t_i]$, for $i=1, 2, \cdots, n$, so that on $J_i$, $\l$ is Lip(1/2) with $|| \l ||_{1/2} < 4$.
Therefore the hull $\hat{K}_i$ generated by $\l$ restricted to $J_i$ is a simple curve by Theorem \ref{thm:cis4thm}.
The concatenation property allows us to put these pieces together to conclude that 
$$K_T = \hat{K}_1 \cup g^{-1}_{t_1}( \hat{K}_2 ) \cup \cdots \cup g^{-1}_{t_{n-1}}( \hat{K}_{n} )$$
is a simple curve. Since this holds for almost every $\omega\in\Omega$, the desired result follows.

To prove (b),  
assume $\tau, \epsilon, c >0$ and $0<\beta < \frac{1}{2 \alpha}$
so that a.s.\ 
$|X_\tau - X_t| \geq c (\tau - t)^\beta$ for all $t \in (\tau-\epsilon, \tau)$.
Set $T=D_{\tau-}$, where $D_{t-}$ denotes the left limit of the path at $t>0$ and $D_{0-}:=D_0=0$.
In other words, the random time $T$ is the first time that $E_t=\tau.$
Note that given $\tau$, the time-reversed process $(D_{\tau-}-D_{(\tau-t)-})_{t\in[0,\tau]}$ is a subordinator having the same distribution as $(D_t)_{t\in[0,\tau]}$.
Since $D$ and $\tau$ are assumed independent, by Lemma \ref{lemma:subbehaviorat0},
for any fixed $\gamma \in [1, 1/\alpha)$,
\begin{align}
	\P\biggl( \lim_{t\downarrow 0} \frac{D_{\tau-}-D_{(\tau-t)-}}{t^\gamma}=0\biggr)
	&=\E\biggl[\P\biggl( \lim_{t\downarrow 0} \frac{D_{\tau-}-D_{(\tau-t)-}}{t^\gamma}=0\biggm| \tau\biggr)\biggr]\\
	&=\E\biggl[\P\biggl( \lim_{t\downarrow 0} \frac{D_t}{t^\gamma}=0\biggm| \tau\biggr)\biggr]
	=1.
\end{align}
Hence, 
for any $m>0$, almost surely, there exists $\delta \in (0, \epsilon)$ such that 
\[
	D_{\tau-}-D_{(\tau-t)-}\le m t^\gamma\ \ \textrm{for all}\ \ t\in[0,\delta]. 
\]
For this particular path, fix $t\in[D_{\tau-\delta},T)$  and set $r = E_t$, which satisfies $r\in[\tau-\delta,\tau]$ and $t\in[D_{r-},D_r)$.
Then the above condition yields
\[
	T-t\le D_{\tau-}-D_{r-}\le m(\tau-r)^\gamma = m (E_T - E_t)^\gamma. 
\]
Therefore, 
\[
	|\lambda(T)-\lambda(t)| 
	\ge c(E_T-E_t)^\beta
	\ge \dfrac{c}{m^{\beta/\gamma}} (T-t)^{\beta/\gamma} \ \ \textrm{for all} \ \ t\in[D_{\tau-\delta},T).
\]
Now, since $\gamma \in [1, 1/\alpha)$ and $m>0$ are arbitrary, we choose $\gamma$ so that $\beta/\gamma = 1/2$, and we choose $m$ satisfying $c/\sqrt{m}\ge 4$. 
An application of Theorem \ref{thm:lindrobinsthm} implies that almost surely $\lambda(t)$ does not generate a simple curve. 
\end{proof}


We now turn to the Brownian case: 

\begin{thm:timechangebmnotsimple}
Let $E$ be the inverse of a subordinator $D$ whose L\'evy measure is infinite and Laplace exponent $\psi$ is regularly varying at $\infty$ with index $\alpha\in(0,1)$. 
Let $B$ be a Brownian motion independent of $D$. 
Then for any $\kappa>0$, almost surely, the time-changed Brownian motion process $\lambda(t) = \kappa B_{E_t}$ does not generate a simple curve.
\end{thm:timechangebmnotsimple}

In order to apply Theorem \ref{thm:timechangesqrtnotsimple}, we wish to control the growth of the time-changed Brownian motion with an appropriate square root function,
 as illustrated in  Figure \ref{fig:sqrtunderBM}. 
 We obtain this control through the relationship between Brownian motion and 3-dimensional Bessel processes, as given in the following two results.
Recall that we say $Y$ is a $d$-dimensional Bessel process if it satisfies $\mathrm{d}Y_t=\frac{a}{Y_t}\mathrm{d}t+\mathrm{d}B_t$ for $a=\frac{d-1}{2}$. For more details about Bessel processes, see Section 1.10 of \cite{lawler}. 
For a continuous real-valued process $X$ and $0<c<\infty$, define
\begin{equation}
    \tau_c^X=\inf\{t>0:X_t=c\}
    \text{ and }
    \sigma_c^X=\sup\{t>0:X_t=c\}.
\end{equation}

\begin{figure}
    \centering
    \includegraphics[width=3in]{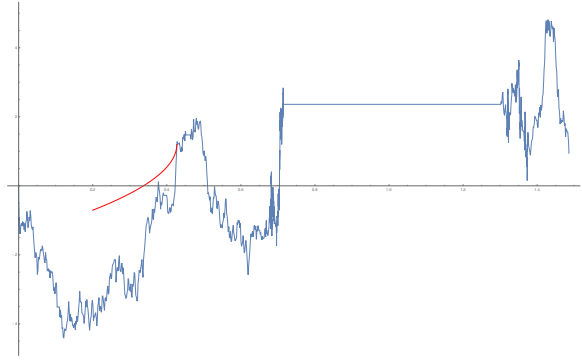}
    \caption{Time-changed Brownian motion sample path (blue) moving faster than $B_{E_T}-4\sqrt{T-t}$ (red) near $T$.}
    \label{fig:sqrtunderBM}
\end{figure}

\begin{prop}[\cite{williams}]\label{prop:williams2}
Let $B$ be a Brownian motion starting at 0, $Y$ a 3-dimensional Bessel process starting at 0, and let $0<c<\infty$. Let $\tau_c=\tau_c^B$. Then the two processes
\begin{equation}
    \{c-B_{\tau_c-t}:0\leq t\leq\tau_c\}
    \text{ and }
    \{Y_t:0\leq t\leq \sigma_c^Y\}
\end{equation}
are identical in distribution.
\end{prop}

\begin{thm}[\cite{shiga}]\label{thm:shigawatanabe}
Let $\varphi(t)\downarrow0$ when $t\downarrow 0$ and let $Y$ be a $d$-dimensional Bessel process. Then for $d\geq 2$,
\begin{equation}
    \P(Y_t<\varphi(t)\sqrt{t}\text{ i.o. }t\downarrow 0)=1\text{ or }0
\end{equation}
according as
\begin{equation}
    \int_{0}^{\infty}\varphi(t)^{d-2}\frac{\mathrm{d}t}{t}=\infty\text{ or }<\infty\quad (d>2);
\end{equation}
\begin{equation}
    \int_{0}^{\infty}\frac{1}{|\log\varphi(t)|}\frac{\mathrm{d}t}{t}=\infty\text{ or }<\infty\quad (d=2).
\end{equation}
\end{thm}

In our proof of Theorem \ref{thm:timechangebmnotsimple}, 
we will use $\varphi(t)=(\log\frac{1}{t})^{-\eta}$ for $\eta>0$ and $d=3$. Then $\varphi(t)\downarrow 0$ when $t\downarrow 0$. Also, $\int_{0}^{\infty}\varphi(t)\frac{\mathrm{d}t}{t}=\infty$ for $\eta\leq 1$ and $\int_{0}^{\infty}\varphi(t)\frac{\mathrm{d}t}{t}<\infty$ for $\eta>1$.

\begin{proof}[Proof of Theorem \ref{thm:timechangebmnotsimple}]

Let $0<c<\infty$ and $\tau_c=\tau_c^B$. Due to Proposition \ref{prop:williams2}, $(c-B_{\tau_c-t})_{0\leq t\leq\tau_c}$ has the same distribution as $(Y_t)_{0\leq t\leq\sigma_c}$ for a 3-dimensional Bessel process $Y$ where $\sigma_c=\sigma_c^Y$. By Theorem \ref{thm:shigawatanabe} and the discussion before this theorem, if $\eta>1$
\begin{equation}
    \P\left(Y_t<\frac{\sqrt{t}}{(\log\frac{1}{t})^\eta}\text{ i.o. }t\downarrow 0\right)=0.
\end{equation}
As such, for $a \in (\frac{1}{2}, \frac{1}{2\alpha})$ and some $h>0$, almost surely there exists $\epsilon>0$ so that $Y_t\geq ht^a$ for $t\in[0,\epsilon]$. 
Therefore, almost surely there exists $\epsilon>0$ with $\kappa B_{\tau_c}- \kappa B_t\geq \kappa h(\tau_c-t)^a$ for $t\in[\tau_c-\epsilon,\tau_c]$.  
Since $B_t$ is assumed independent of $D_t$, the random time $\tau_c=\tau_c^B$ is also independent of $D_t$.
 An application of  Theorem \ref{thm:timechangesqrtnotsimple}(b) completes the proof. 
\end{proof}

\subsection{Examples of Time-Changed Deterministic Functions}
\label{deterministicexample}

In this section we consider two deterministic functions $\phi$ whose Loewner hulls have previously been analyzed.  
To highlight the effect of the time change on driving functions, we look at the behavior of the hulls driven by $\lambda(t)=\phi(E_t)$, where $E$ is an inverse $\alpha$-stable subordinator.

\smallskip

\noindent {\bf Example 1}: The Loewner hulls driven by $\phi(t) = c \sqrt{t}$
 are line-segments starting from 0, with an angle determined by the constant $c$. See \cite{knk}.
 In particular, the hulls are always simple curves. 
With the driving function $\lambda(t) = c \sqrt{E_t}$,
we see different hull behavior in two regimes, when $\alpha >1/2$ and when $\alpha < 1/2$.  
See Figure \ref{fig:sqrtEt}. 

\begin{figure}
    \centering
    \includegraphics[width=3in]{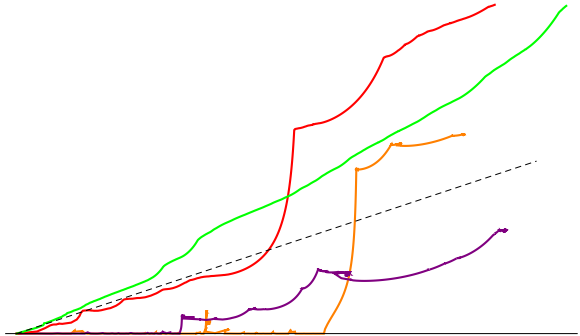}
    \caption{Sample hulls generated by  $\sqrt{E_t}$, where $E_t$ is an inverse $\alpha$-stable subordinator ($\alpha = 0.9 $ in green, $\alpha = 0.7$ in red, $\alpha = 0.4$ in purple and $\alpha = 0.3$ in orange). The black dashed line represents a hull generated by $\sqrt{t}$ (without a time change).}
    \label{fig:sqrtEt}
\end{figure}

Suppose first that $\alpha > 1/2$.  
Let $K_t$ be the hull generated by $\lambda$, let $\epsilon>0$, and let $\tau_\e$ be the first time that $E_t = \e$.
We first consider $\lambda$ restricted to the interval $[\tau_\e,\infty)$ 
which generates the hull $g_{\tau_\e}(K_t \setminus K_{\tau_\e})$.
Since $\phi$ is $C^1$ on $[\e, \infty)$, Theorem \ref{thm:timechangesqrtnotsimple}(a)
implies that  $g_{\tau_\e}(K_t \setminus K_{\tau_\e})$ is a simple curve.  
Thus  $K_t \setminus K_{\tau_\e}$ must be a simple curve for $t> \tau_\e$,
and so we can define $\gamma(t)$, for $t \in (0, \infty)$,  so that $K_b \setminus K_a = \gamma[a, b]$.
Further, using the continuity of $\l$ we can argue that $\gamma(t) \to 0$ as $t \to 0^+$.  
Thus almost surely $\gamma$ is a simple curve defined on $[0, \infty)$, and $K_t=\g[0,t]$.
Proposition \ref{tangentialdeparture} and Lemma \ref{lemma:subbehaviorat0} show that $\gamma$ leaves the real line tangentially.

Suppose next that $\alpha < 1/2$.  
Arguing as above shows that the hull must still leave the real line tangentially.  
However the hull is no longer a simple curve.
In particular, we can apply Theorem \ref{thm:timechangesqrtnotsimple}(b) with $\beta = 1$ for any fixed time $\tau > 0$.
This shows that the hulls generated by $\lambda$ are far from being simple curves, as they are non-simple at the first time that $E_t$ reaches height $\tau.$  Since this is true for all $\tau >0$ and since $E_t$ is almost surely not constant in a neighborhood of 0 (otherwise $D_0$ would be positive, which contradicts the assumption $D_0=0$), we find that almost surely $K_t$ is non-simple for all $t>0$.

\smallskip

\noindent {\bf Example 2}: For our last example, we consider the case when $\lambda(t) = c\, W(E_t)$, 
where $  W(t) = \sum_{n=0}^\infty 2^{-n/2} \cos(2^n t) $ is the Weierstrass function.  
Since this function is continuous but nowhere differentiable, it serves as a deterministic analogue of Brownian motion.
Similar to SLE,  the Loewner hulls driven by $c\,W$, studied in \cite{lindrobins}, have a phase transition.  
In particular for $c$ small enough $c\,W$ generates simple curves, but for $c$ large enough the hulls are not simple.
In \cite{lindrobins}, it is also shown that near a local maximum $W_t$ grows faster than an appropriately scaled square root curve.  Therefore, Theorem \ref{thm:timechangesqrtnotsimple}(b) implies that $c\, W(E_t)$ does not generate a simple curve.  Since these local maximums are dense, we can further conclude that almost surely the Loewner hull $K_t$ is non-simple for all $t>0$.
See Figure \ref{fig:WEt}.

\begin{figure}
    \centering
    \includegraphics[width=3in]{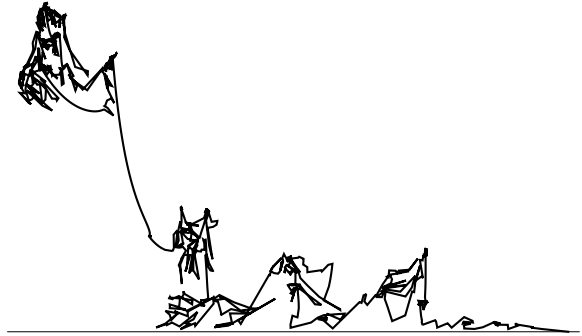}
    \caption{Sample hull generated by  $W(E_t)$ where $E_t$ is an inverse $0.7$-stable subordinator. }
    \label{fig:WEt}
\end{figure}

%
%
\section*{Acknowledgments} 
The authors thank Krzysztof Burdzy for pointing out some useful references. Part of this research was conducted while Kei Kobayashi and Andrew Starnes were affiliated with the Department of Mathematics of the University of Tennessee. The authors thank the University of Tennessee for their support.

\end{document}